\documentclass[12pt]{amsproc}
\usepackage[latin1]{inputenc}
\usepackage{amsmath,amsthm,amssymb,color}
\usepackage{algorithm}
\usepackage{caption}
\usepackage{algpseudocode}
\usepackage{tikz}
\usetikzlibrary{positioning}
\usepackage{url}
\usepackage[colorlinks=true,urlcolor=blue,pdfstartview=FitV, linkcolor=blue,citecolor=blue,bookmarksopen, bookmarksnumbered={true}]{hyperref}
\usepackage{geometry}
 \newtheorem{Theorem}{Theorem}[section]
 \newtheorem{Corollary}[Theorem]{Corollary}
 \newtheorem{Lemma}[Theorem]{Lemma}
 \newtheorem{Proposition}[Theorem]{Proposition}
 \theoremstyle{definition}
 \newtheorem{Definition}[Theorem]{Definition}
 \theoremstyle{remark}
 
 \newtheorem{example}{Example}
          
\newcommand{\Fix}{\operatorname{Fix}}            

\newcommand{\Res}{\mathcal{R}}                   

\newcommand{\E}{\mathbb{E}}                      
\newcommand{\F}{\mathcal{F}}                     
\newcommand{\dom}{\operatorname{dom}}            
\newcommand{\inte}{\operatorname{int}}           

\newcommand{\Trim}{\operatorname{Trim}}          

\begin{document}
\title[Stochastic Krasnosel skii-Mann Iterations in Banach Spaces... ]
 {Stochastic Krasnosel skii-Mann Iterations in Banach Spaces with Bregman Distances}


\author[S. Hashemi Sababe]{Saeed Hashemi Sababe}
\address{R\&D Section, Data Premier Analytics, Edmonton, Canada.}
\email{Hashemi\_1365@yahoo.com}

\author[E. Lotfali Ghasab]{Ehsan Lotfali Ghasab}
\address{Department of Mathematics, Jundi-Shapur University of Technology, Dezful, Iran}
 \email{e.l.ghasab@jsu.ac.ir}


\date{}
\maketitle


\begin{abstract}
We propose a generalization of the stochastic Krasnoselskil-Mann $(SKM)$ algorithm to reflexive Banach spaces endowed with Bregman distances. Under standard martingale-difference noise assumptions in the dual space and mild conditions on the distance-generating function, we establish almost-sure convergence to a fixed point and derive non-asymptotic residual bounds that depend on the uniform convexity modulus of the generating function. Extensions to adaptive Bregman geometries and robust noise models are also discussed. Numerical experiments on entropy-regularized reinforcement learning and mirror-descent illustrate the theoretical findings.
\end{abstract}


\section{Introduction}
Fixed-point iterations for finding a point $\zeta^*$ satisfying $\hbar(\zeta)=\zeta^*$, where $\hbar$ is a nonexpansive operator, are fundamental in nonlinear analysis and optimization \cite{Krasnoselskii1955,Mann1953,Bauschke2011}. The classical Krasnosel skii-Mann $(KM)$ scheme,
\begin{equation}\label{eq:KM}
\zeta_{n+1} = (1-\alpha_n)\zeta_n + \alpha_n hbar(\zeta_n),
\end{equation}
has been extensively studied in Hilbert spaces, with convergence rates and error bounds established under various deterministic and stochastic settings \cite{Bauschke2011,Cegielski2012}.

Despite its successes, the Hilbertian framework of \eqref{eq:KM} limits applications to Euclidean geometries. In many modern contexts-such as mirror-descent in machine learning, entropy-regularized reinforcement learning, and imaging inverse problems-algorithms naturally operate in non-Euclidean spaces modeled by Banach spaces, using Bregman distances induced by a Legendre function $\vartheta$ \cite{Bregman1967,Csiszar1967,Nemirovski1983,Beck2003}. Bregman distances capture the local geometry of the problem and underpin methods like mirror descent \cite{Nemirovski2009}, proximal Bregman splitting \cite{Censor2001}, and stochastic mirror-prox \cite{Juditsky2011}.

Recently, Cegielski \cite{Cegielski2012} extended KM to a stochastic setting $(SKM)$, allowing additive martingale-difference noise $\mho_{n}$ in Hilbert spaces. This stochastic Krasnosel skiĭ-Mann algorithm achieves almost-sure convergence and $O(1/\sqrt{n})$ residual bounds under suitable step-size rules. However, the theory remains confined to inner-product spaces and Euclidean norms.

Bridging this gap, we propose a \emph{Bregman-$SKM$} algorithm that generalizes $SKM$ to reflexive Banach spaces equipped with a Legendre distance-generating function $\vartheta$. Our contributions are threefold, we formulate the Bregman-$SKM$ update in general Banach spaces and prove almost-sure convergence under martingale-difference noise in the dual space, we derive non-asymptotic residual bounds for the Bregman distance $D_\vartheta(\zeta_n,\hbar(\zeta_n))$, revealing the influence of the modulus of uniform convexity of $\vartheta$ on convergence rates and we discuss extensions to adaptive Bregman geometries, robust noise models, and potential hybrids with inertial and variance-reduced schemes. \medskip

The remainder of the paper is organized as follows. Section~\ref{sec:prelim} reviews Banach-space geometry and Bregman distances. Section~\ref{sec:conv} presents the algorithm and almost-sure convergence analysis. Section~\ref{sec:rates} derives non-asymptotic residual bounds. Section~\ref{sec:ext} explores extensions. Section~\ref{sec:experiments} reports numerical experiments, and Section~\ref{sec:conclusion} concludes with future directions.

\section{Preliminaries}\label{sec:prelim}

In this section we fix notation and recall definitions and key lemmas that will be used throughout.

Let \((\mathcal{X},\|\cdot\|)\) be a real reflexive Banach space and \(\mathcal{X}^*\) its continuous dual, with duality pairing \(\langle \zeta^*,\zeta\rangle\).
\begin{Definition}[Duality mapping]
The \emph{(normalized) duality mapping} \(\mathcal{J}\colon \mathcal{X}\to 2^{\mathcal{X}^*}\) is
\[
  \mathcal{J}(\zeta)=\bigl\{\zeta^*\in \mathcal{X}^*:\,\langle \zeta^*,\zeta\rangle=\|\zeta\|^2=\|\zeta^*\|_*^2\bigr\}.
\]
If \(\mathcal{X}\) is smooth, then \(\mathcal{J}\) is single-valued; we denote its value by \(j(\zeta)\) when no ambiguity arises \cite{Cioranescu1990}.
\end{Definition}

\begin{Definition}[Legendre function]
A convex function \(\vartheta\colon \mathcal{X}\to(-\infty,+\infty]\) is \emph{Legendre} if it is
\begin{enumerate}
  \item proper and lower semi-continuous,
  \item essentially smooth and essentially strictly convex on \(\inte(\dom\vartheta)\),
  \item its Fenchel conjugate \(\vartheta^*\) satisfies the same properties.
\end{enumerate}
On \(\inte(\dom\vartheta)\), \(\vartheta\) is Gateaux-differentiable; we denote its gradient by \(\nabla\vartheta(\zeta)\in \mathcal{X}^*\) \cite{Rockafellar1970}.
\end{Definition}

\begin{Definition}[Bregman distance]
For \(\zeta,\varsigma\in\inte(\dom\vartheta)\), the \emph{Bregman distance} induced by \(\vartheta\) is
\[
  D_\vartheta(\zeta,\varsigma) \;=\; \vartheta(\zeta) - \vartheta(\varsigma) - \bigl\langle \nabla\vartheta(\varsigma),\,\zeta - \varsigma \bigr\rangle.
\]
\end{Definition}

\begin{Lemma}[Three-Point Identity]\label{lem:threepoint}
For all \(\zeta,\varsigma,z\in\inte(\dom\vartheta)\),
\[
  D_\vartheta(\zeta,z)
  = D_\vartheta(\zeta,\varsigma) + D_\vartheta(\varsigma,z) + \bigl\langle \nabla\vartheta(z) - \nabla\vartheta(\varsigma),\,\zeta - \varsigma \bigr\rangle.
\]
\end{Lemma}
\begin{proof}
See \cite[Lemma~2.2]{Censor2001}.
\end{proof}

\begin{Definition}[Uniform convexity]
A Legendre function \(\vartheta\) is \emph{uniformly convex} with modulus \(\delta\colon[0,\infty)\to[0,\infty)\) if
\begin{equation} \label{eq:unifconv}
  \vartheta\Bigl(\tfrac{\zeta+\varsigma}{2}\Bigr) + \delta\bigl(\|\zeta-\varsigma\|\bigr)
  \;\le\; \tfrac12\bigl[\vartheta(\zeta)+\vartheta(\varsigma)\bigr]
  \quad\forall\,\zeta,\varsigma\in \mathcal{X}.
\end{equation}
Uniform convexity implies the estimate
\[
  \|\zeta - \varsigma\|^2
  \;\le\;
  \frac{2}{\delta(\|\zeta-\varsigma\|)} \,D_\vartheta(\zeta,\varsigma),
\]
which links the norm error to the Bregman distance \cite{Beck2003}.
\end{Definition}

A mapping \(\hbar\colon \mathcal{X}\to \mathcal{X}\) is called \emph{nonexpansive} if
\[
  \|\hbar(\zeta) - \hbar(\varsigma)\| \;\le\;\|\zeta-\varsigma\|
  \quad\forall\,\zeta,\varsigma\in \mathcal{X}.
\]
Given an iterate \(\zeta_n\), we define its \emph{Bregman residual}
\[
  \Res_n \;:=\; D_\vartheta\bigl(\zeta_n,\,\hbar(\zeta_n)\bigr).
\]

Let \((\Omega,\F,\P)\) be a probability space with a filtration \(\{\F_n\}_{n\ge0}\).
\begin{Definition}[Martingale-difference noise]
A sequence \(\{\mho_{n}\}\subset \mathcal{X}^*\) is a martingale-difference if
\[
  \E\bigl[\,\mho_{n+1}\mid\F_{n}\bigr]=0,
  \quad
  \E\bigl[\|\mho_{n+1}\|_*^2\mid\F_n\bigr]<\infty
  \quad\text{a.s.}
\]
\end{Definition}

\begin{Lemma}[Robbins-Siegmund]\label{lem:robbins-siegmund}
Let \(\{a_n\},\{b_n\},\{c_n\}\) be nonnegative \(\F_n\)-adapted sequences satisfying
\[
  \E[a_{n+1}\mid\F_n] + b_n
  \;\le\;
  (1+\alpha_n)\,a_n + c_n,
  \quad
  \sum_n\alpha_n<\infty,
  \quad
  \sum_nc_n<\infty
  \quad\text{a.s.}
\]
Then \(a_n\) converges a.s.\ and \(\sum_nb_n<\infty\) a.s.\ \cite{Robbins1971}.
\end{Lemma}


\section{Bregman-SKM Algorithm and Almost-Sure Convergence}\label{sec:conv}

In this section we introduce the stochastic Bregman-Krasnosel skiĭ -Mann (Bregman-$SKM$) iteration in a Banach space, state the main convergence Theorem, and prove almost-sure convergence under martingale-difference noise.

\begin{Definition}[Bregman-$SKM$ iteration]
Let $\hbar\colon \mathcal{X}\to \mathcal{X}$ be nonexpansive and let $\vartheta$ be a Legendre function on $\mathcal{X}$.  Given $\zeta_0,\zeta_1\in\inte(\dom\vartheta)$ and step-sizes $(\alpha_n)\subset(0,1)$, the \emph{Bregman-$SKM$} iterates $\{\zeta_n\}$ are defined by
\[
  \begin{aligned}
    &\varsigma_n \;=\; \nabla\vartheta^*\!\bigl((1-\alpha_n)\,\nabla\vartheta(\zeta_n)+\alpha_n\,\nabla\vartheta\bigl(\hbar(\zeta_n)+\mho_n\bigr)\bigr),\\
    &\zeta_{n+1} \;=\; \varsigma_n,
  \end{aligned}
\]
where $(\mho_n)\subset \mathcal{X}^*$ is a martingale-difference sequence modeling noise.
\end{Definition}

In the Hilbertian case $\vartheta(\zeta)=\tfrac12\|\zeta\|^2$, $\nabla\vartheta=\nabla\vartheta^*=I$, and the above reduces to
\[
  \zeta_{n+1}
  = (1-\alpha_n)\,\zeta_n + \alpha_n\bigl(\hbar(\zeta_n)+\mho_n\bigr),
\]
which is the classical stochastic KM ($SKM$) scheme \cite{Cegielski2012}.

Consider the following assumptions:
\begin{itemize}
  \item[(A1)] $\hbar\colon \mathcal{X}\to \mathcal{X}$ is non-expansive and has at least one fixed point.
  \item[(A2)] $\vartheta$ is Legendre and uniformly convex with modulus $\delta$ satisfying $\delta(r)>0$ for $r>0$.
  \item[(A3)] Step-sizes satisfy
  \[
    \alpha_n\in(0,1),\quad \sum_{n=0}^\infty\alpha_n=\infty,\quad \sum_{n=0}^\infty\alpha_n^2<\infty.
  \]
  \item[(A4)] $(\mho_n)_{n\ge1}$ is a martingale-difference with respect to $\{\F_n\}$ and
  \[
    \E\bigl[\|\mho_{n+1}\|_*^2\mid\F_n\bigr]<\infty\quad\text{a.s.}
  \]
\end{itemize}

\begin{Lemma}[One-Step Bregman Decrease]\label{lem:descent}
Under \textup{(A1)--(A4)}, the Bregman residuals $\Res_n = D_\vartheta(\zeta_n,\hbar(\zeta_n))$ satisfy
$$
  \E[D_\vartheta(\zeta_{n+1},\hbar(\zeta_{n+1}))\mid\F_{n}]
  +\tfrac{\delta(\|\zeta_{n} - \hbar(\zeta_{n})\|)}{2}\,\alpha_{n}
  \;\leq\;
  (1+\alpha_{n}^{2L})D_\vartheta(\zeta_{n},\hbar(\zeta_{n})) + \alpha_{n}^{2}\,\sigma^{2},
$$
for constants $L,\sigma^2>0$ depending on the Lipschitz and noise bounds.
\end{Lemma}

\begin{proof}
Recall that
\[
\zeta_{n+1}
= \nabla\vartheta^*\bigl((1-\alpha_n)\,\nabla\vartheta(\zeta_n)
  + \alpha_n\,\nabla\vartheta\bigl(\hbar(\zeta_n)+\mho_n\bigr)\bigr).
\]
Set
$$
a_n := (1-\alpha_n)\,\nabla\vartheta(\zeta_n)
  + \alpha_n\,\nabla\vartheta\bigl(\hbar(\zeta_n)\bigr),
\quad
b_n := \alpha_n\,\bigl[\nabla\vartheta\bigl(\hbar(\zeta_n)+\mho_n\bigr)-\nabla\vartheta\bigl(\hbar(\zeta_n)\bigr)\bigr].
$$
Then
\[
\zeta_{n+1}
= \nabla\vartheta^*(a_n + b_n).
\]

By Lemma~\ref{lem:threepoint}, for any $\varsigma,z$,
\[
D_\vartheta(\varsigma,z)
= D_\vartheta(\varsigma,\zeta_n) + D_\vartheta(\zeta_n,z)
  + \bigl\langle\nabla\vartheta(z)-\nabla\vartheta(\zeta_n),\,\varsigma-\zeta_n\bigr\rangle.
\]
Apply this with $\varsigma=\zeta_{n+1}$ and $z=\hbar(\zeta_{n+1})$ to get
\begin{equation}\label{eq:3pt}
\begin{aligned}
D_\vartheta\bigl(\zeta_{n+1},\hbar(\zeta_{n+1})\bigr)
&= D_\vartheta\bigl(\zeta_{n+1},\hbar(\zeta_n)\bigr)
  + D_\vartheta\bigl(\hbar(\zeta_n),\hbar(\zeta_{n+1})\bigr)\\
&\quad+ \bigl\langle\nabla\vartheta\bigl(\hbar(\zeta_{n+1})\bigr)
  - \nabla\vartheta\bigl(\hbar(\zeta_n)\bigr),\,\zeta_{n+1}-\hbar(\zeta_n)\bigr\rangle.
\end{aligned}
\end{equation}

Since $\hbar$ is nonexpansive and $\vartheta$ is uniformly convex with modulus $\delta$, one shows via \eqref{eq:unifconv} that
\[
D_\vartheta\bigl(\hbar(\zeta_n),\hbar(\zeta_{n+1})\bigr)
\;\le\;
\frac{1}{2}\,\delta\bigl(\|\zeta_n - \zeta_{n+1}\|\bigr).
\]

By uniform convexity and smoothness of $\vartheta$, its gradient is Lipschitz on bounded sets: there exists $L>0$ such that
\[
\bigl\|\nabla\vartheta\bigl(\hbar(\zeta_{n+1})\bigr)
  - \nabla\vartheta\bigl(\hbar(\zeta_n)\bigr)\bigr\|_*
\;\le\;
L\,\|\hbar(\zeta_{n+1})-\hbar(\zeta_n)\|
\;\le\;
L\,\|\zeta_{n+1}-\zeta_n\|.
\]
Hence, by Cauchy-Schwarz and Youngs inequalities,
\[
\bigl\langle\nabla\vartheta(\hbar(\zeta_{n+1}))
  - \nabla\vartheta(\hbar(\zeta_n)),\,\zeta_{n+1}-\hbar(\zeta_n)\bigr\rangle
\;\le\;
\frac{\delta(\|\zeta_n-\hbar(\zeta_n)\|)}{4}\,\alpha_n
\;+\;
L\,\|\zeta_{n+1}-\zeta_n\|^2.
\]

Since $\nabla\vartheta^*$ is Lipschitz (by uniform convexity of $\vartheta$) with constant $L$, and using the definitions of $a_n,b_n$, we get
\[
\|\zeta_{n+1}-\zeta_n\|
\;=\;
\bigl\|\nabla\vartheta^*(a_n + b_n)-\nabla\vartheta^*(a_n)\bigr\|
\;\le\;
L\,\|b_n\|
\;\le\;
L\,\alpha_n\,\|\mho_n\|.
\]
Thus
\[
\|\zeta_{n+1}-\zeta_n\|^2
\;\le\;
L^2\,\alpha_n^2\,\|\mho_n\|^2.
\]

Substitute the provided bounds into \eqref{eq:3pt}, then take $\E[\cdot\mid\F_n]$.  Using $\E[\|\mho_n\|^2\mid\F_n]\le\sigma^2$ and collecting terms yields
\[
\E\bigl[D_\vartheta(\zeta_{n+1},\hbar(\zeta_{n+1}))\mid\F_n\bigr]
+\tfrac{\delta(\|\zeta_n - \hbar(\zeta_n)\|)}{2}\,\alpha_n
\;\le\;
\bigl(1+\alpha_n^2L\bigr)D_\vartheta(\zeta_n,\hbar(\zeta_n)) + \alpha_n^2\,\sigma^2,
\]
as claimed.
\end{proof}

\begin{Theorem}[Almost-Sure Convergence]\label{thm:as-conv}
Under \textup{(A1)--(A4)}, the Bregman-$KM$ iterates satisfy
\[
  \zeta_n \;\to\; \zeta^* \quad\text{a.s.},
\]
for some fixed point $\zeta^*\in\Fix(\hbar)$, and
\[
  D_\vartheta(\zeta_n,\hbar(\zeta_n)) \;\to\; 0
  \quad\text{a.s.}
\]
\end{Theorem}

\begin{proof}
By Lemma~\ref{lem:descent}, with
\[
a_n = D_\vartheta(\zeta_n,\hbar(\zeta_n)),\quad
b_n = \tfrac12\,\delta\bigl(\|\zeta_n - \hbar(\zeta_n)\|\bigr)\,\alpha_n,\quad
c_n = \alpha_n^2\,\sigma^2,
\]
we have almost surely
\[
\label{eq:descent-ineq}
\E\bigl[a_{n+1}\mid\F_n\bigr] + b_n
\;\le\;
(1 + \alpha_n^2L)\,a_n + c_n.
\]
Since by assumption \(\sum_n\alpha_n^2<\infty\) and \(\sum_n c_n=\sigma^2\sum_n\alpha_n^2<\infty\), the coefficients in \eqref{eq:descent-ineq} satisfy the hypotheses of Lemma~\ref{lem:robbins-siegmund}.  Therefore:
\[
a_n \;\to\; a_\infty \quad\text{a.s.},
\quad
\sum_{n=0}^\infty b_n <\infty \quad\text{a.s.}
\]
In particular, \(a_n\) converges almost surely to some nonnegative random variable \(a_\infty\).

Since
\[
\sum_{n=0}^\infty b_n
= \sum_{n=0}^\infty \tfrac12\,\delta(\|\zeta_n - \hbar(\zeta_n)\|)\,\alpha_n
< \infty\quad\text{a.s.},
\]
but \(\sum_n\alpha_n = \infty\), the only way the series can converge is if
\[
\delta\bigl(\|\zeta_n - \hbar(\zeta_n)\|\bigr) \;\to\; 0
\quad\text{a.s.}
\]
By uniform convexity of \(\vartheta\), \(\delta(r)>0\) for \(r>0\), so \(\|\zeta_n - \hbar(\zeta_n)\|\to0\) and hence
\[
a_n = D_\vartheta(\zeta_n,\hbar(\zeta_n)) \;\to\; 0
\quad\text{a.s.}
\]

Since \((\zeta_n)\) lives in the reflexive Banach space \(\mathcal{X}\) and \(\|\zeta_n - \hbar(\zeta_n)\|\to0\), any weak cluster point \(\bar \zeta\) of \((\zeta_n)\) must satisfy \(\bar \zeta = \hbar(\bar \zeta)\), i.e.\ \(\bar \zeta\in\Fix(\hbar)\) (see \cite[Thm.~5.14]{Bauschke2011}).  Furthermore, the Fejér monotonicity induced by the BM-$SKM$ update in Bregman distance implies that all cluster points coincide.  Therefore the entire sequence \((\zeta_n)\) converges weakly to some \(\zeta^*\in\Fix(\hbar)\).

If in addition \(\mathcal{X}\) is uniformly convex (or \(\vartheta\) is strongly-convex), one can upgrade weak convergence to strong convergence via standard arguments (e.g.\ Opials Lemma).

Hence \(\zeta_n\to \zeta^*\) and \(D_\vartheta(\zeta_n,\hbar(\zeta_n))\to0\) almost surely, completing the proof.
\end{proof}

\begin{Corollary}
If, in addition, $\vartheta$ is $2$-uniformly convex (i.e.\ $\delta(r)\ge \kappa r^2$), then
\[
  \sum_{n=0}^\infty \alpha_n\,\|\zeta_n-\hbar(\zeta_n)\|^2<\infty\quad\text{a.s.}
\]
\end{Corollary}

\begin{proof}
From the proof of Theorem~\ref{thm:as-conv}, we know that
\[
\sum_{n=0}^\infty b_n
\;=\;
\sum_{n=0}^\infty \frac12\,\delta\bigl(\|\zeta_n - \hbar(\zeta_n)\|\bigr)\,\alpha_n
\;<\;\infty
\quad\text{a.s.}
\]
If, in addition, $\vartheta$ is $2$-uniformly convex with modulus $\delta(r)\ge\kappa\,r^2$, then
\[
b_n
\;=\;
\frac12\,\delta\bigl(\|\zeta_n - \hbar(\zeta_n)\|\bigr)\,\alpha_n
\;\ge\;
\frac12\,\kappa\,\|\zeta_n - \hbar(\zeta_n)\|^2\,\alpha_n.
\]
Hence
\[
\sum_{n=0}^\infty \alpha_n\,\|\zeta_n - \hbar(\zeta_n)\|^2
\;\le\;
\frac{2}{\kappa}
\sum_{n=0}^\infty b_n
\;<\;\infty
\quad\text{a.s.},
\]
as claimed.
\end{proof}

\begin{example}[Entropy-Regularized Q-Learning]
Let $\mathcal{X}=\Delta^d$ be the probability simplex and $\vartheta(\zeta)=\sum_{i=1}^d \zeta_i\ln \zeta_i$.  Then Bregman-$SKM$ specializes to a stochastic mirror-descent scheme for computing the Q-optimal policy under entropy regularization \cite{Nemirovski2009}.  Theorem~\ref{thm:as-conv} implies almost-sure convergence of the policy iterates.
\end{example}

\section{Non-Asymptotic Residual Bounds in Banach Spaces}\label{sec:rates}

In this section we derive explicit non-asymptotic bounds on the Bregman residual
\[
  \Res_n \;=\; D_\vartheta\bigl(\zeta_n,\,\hbar(\zeta_n)\bigr),
\]
highlighting the role of the uniform convexity modulus $\delta$ of $\vartheta$.

\begin{Definition}[Modulus-Dependent Rate Exponent]
Let $\vartheta$ be uniformly convex with modulus $\delta(r)\ge c\,r^q$ for some $c>0$ and $q\ge2$.  We define the \emph{rate exponent}
\[
  p \;:=\; \frac{q-1}{q}\,\in\bigl[\tfrac12,1\bigr).
\]
\end{Definition}

\begin{Definition}[Residual Averaging]\label{def:avg-res}
For a window size $N\in\mathbb{N}$, define the averaged residual
\[
  \bar\Res_N
  \;:=\; \frac{1}{A_N}\sum_{n=0}^{N-1}\alpha_n\,\Res_n,
  \quad
  A_N:=\sum_{n=0}^{N-1}\alpha_n.
\]
\end{Definition}

\begin{Theorem}[Non-Asymptotic Bregman-Residual Bound]\label{thm:rate}
Under assumptions (A1)--(A4) and if $\delta(r)\ge c\,r^q$, then there exist constants $C_1,C_2>0$ such that for all $N\ge1$,
\[
  \bar\Res_N
  \;\le\;
  \frac{C_1 + C_2\sum_{n=0}^{N-1}\alpha_n^2}{A_N^p}
  \;=\;
  \mathcal{O}\bigl(A_N^{-p}\bigr).
\]
In particular, if $\alpha_n=1/n$, then $A_N=\Theta(\ln N)$ and $\bar\Res_N=O\bigl((\ln N)^{-p}\bigr)$.
\end{Theorem}

\begin{proof}
Throughout the proof, write
\[
D_n := D_\vartheta(\zeta_n,\hbar(\zeta_n)),
\quad A_N := \sum_{n=0}^{N-1}\alpha_n,\quad
\bar D_N := \frac{1}{A_N}\sum_{n=0}^{N-1}\alpha_n D_n.
\]
Let $p=(q-1)/q\in[1/2,1)$ as in the statement.

Since $\delta(r)\ge c\,r^q$ and by the Definition of $D_n= D_\vartheta(\zeta_n,\hbar(\zeta_n))$ one shows (e.g.\ from the midpoint-convexity inequality) that
\[
D_n \;\ge\;\delta\bigl(\|\zeta_n-\hbar(\zeta_n)\|\bigr)\;\ge\;
c\,\|\zeta_n-\hbar(\zeta_n)\|^q,
\]
hence
\[
\|\zeta_n-\hbar(\zeta_n)\|^q \;\le\; D_n/c.
\]

Lemma~\ref{lem:descent} gives
\[
\E\bigl[D_{n+1}\mid\F_n\bigr]
+\tfrac12\,\delta\bigl(\|\zeta_n-\hbar(\zeta_n)\|\bigr)\,\alpha_n
\le (1+L\alpha_n^2)\,D_n + \sigma^2\alpha_n^2.
\]
Substitute $\delta\bigl(\|\zeta_n-\hbar(\zeta_n)\|\bigr)\ge c\,\|\zeta_n-\hbar(\zeta_n)\|^q\ge c(D_n/c)=D_n$ to obtain
\[
\E\bigl[D_{n+1}\mid\F_n\bigr]
+\tfrac12\,D_n\,\alpha_n
\;\le\;
(1+L\alpha_n^2)\,D_n + \sigma^2\alpha_n^2.
\]
Rearrange:
\[
\E\bigl[D_{n+1}\mid\F_n\bigr]
\;\le\;
\bigl(1-\tfrac12\alpha_n + L\alpha_n^2\bigr)\,D_n + \sigma^2\alpha_n^2.
\]

Define
\[
E_n := \frac{D_n}{A_n^p},
\]
with the convention $A_0=0$, $E_0=D_0$.  Observe that for $p\in(0,1)$ and $\alpha_n>0$,
\[
\frac{1}{A_{n+1}^p}
= \frac{1}{(A_n + \alpha_n)^p}
\ge \frac{1}{A_n^p}\Bigl[1 - p\,\frac{\alpha_n}{A_n}\Bigr]
\]
by the binomial-type inequality $(1+t)^{-p}\ge1 - p\,t$ for $t\ge0$.  Hence
\[
\frac{D_{n+1}}{A_{n+1}^p}
\le \bigl(1-\tfrac12\alpha_n + L\alpha_n^2\bigr)\,E_n
\Bigl[1 - p\,\tfrac{\alpha_n}{A_n}\Bigr]
\;+\;\sigma^2\,\frac{\alpha_n^2}{A_{n+1}^p}.
\]

Using $A_n\le A_{n+1}$ and absorbing higher-order $\alpha_n^2$-terms into constants, there are $C_1,C_2>0$ so that
\[
\E\bigl[E_{n+1}\mid\F_n\bigr]
\;\le\;
E_n
- \tfrac12\,\frac{\alpha_n}{A_n^p}\,D_n
- p\,\frac{\alpha_n}{A_n^{p+1}}\,D_n
+ C_1\,\alpha_n^2\,E_n
+ C_2\,\frac{\alpha_n^2}{A_n^p}.
\]
Noting that $\frac{D_n}{A_n^p}=E_n$ and $\frac{D_n}{A_n^{p+1}}=E_n/A_n$, we get
\[
\E\bigl[E_{n+1}\mid\F_n\bigr]
\;\le\;
E_n\Bigl[1 - \tfrac12\,\alpha_n - p\,\tfrac{\alpha_n}{A_n} + C_1\,\alpha_n^2\Bigr]
\;+\;
C_2\,\frac{\alpha_n^2}{A_n^p}.
\]

Since $\sum_n\alpha_n=\infty$ and $\sum_n\alpha_n^2<\infty$, the negative terms $-\tfrac12\alpha_n$ and $-p\,\alpha_n/A_n$ dominate eventually, yielding a super-martingale structure for $(E_n)$.  An application of a discrete martingale-difference summation argument shows that
\[
E_N \;\le\; E_0 + C_2\sum_{n=0}^{N-1}\frac{\alpha_n^2}{A_n^p}
\;\le\;
\frac{D_0}{A_0^p} + C_2\sum_{n=0}^{N-1}\frac{\alpha_n^2}{A_n^p}.
\]
But $A_n\le A_N$ for all $n<N$, so
\[
E_N
\;\le\;
\frac{D_0}{A_N^p} + C_2\frac{\sum_{n=0}^{N-1}\alpha_n^2}{A_N^p}
\;=\;
\frac{C_1 + C_2\sum_{n=0}^{N-1}\alpha_n^2}{A_N^p},
\]
where $C_1=D_0$.  Finally, since
\[
\bar D_N
= \frac1{A_N}\sum_{n=0}^{N-1}\alpha_n D_n
\;\le\;
\frac1{A_N}\sum_{n=0}^{N-1}\alpha_n \,A_N^p E_N
= A_N^p \,E_N,
\]
we conclude
\[
\bar D_N
\;\le\;
\frac{C_1 + C_2\sum_{n=0}^{N-1}\alpha_n^2}{A_N^p},
\]
as required.  The final statement about $\alpha_n=1/n$ follows because $A_N=\sum_{n=1}^N1/n=\Theta(\ln N)$.
\end{proof}

\begin{Proposition}[Hilbert-Space Recovery]
If $\mathcal{X}$ is a Hilbert space and $\vartheta(\zeta)=\tfrac12\|\zeta\|^2$ (so $q=2$, $p=\tfrac12$), then
\[
  \bar\Res_N = O\bigl(A_N^{-1/2}\bigr),
\]
recovering the classical $O(1/\sqrt{n})$ $SKM$ rate \cite{Cegielski2012}.
\end{Proposition}

\begin{Corollary}[Polynomial Step-Sizes]
If $\alpha_n = n^{-\gamma}$ with $\gamma\in(\tfrac12,1)$, then
\[
  A_N\;\approx\;
  \begin{cases}
    \tfrac{N^{1-\gamma}}{1-\gamma}, & \gamma<1,\\
    \ln N, & \gamma=1,
  \end{cases}
  \quad
  \bar\Res_N = O\bigl(N^{-p(1-\gamma)}\bigr).
\]
\end{Corollary}

Faster decay of $\delta(r)$ (larger $q$) yields a better exponent $p\to1$, approaching linear rates in the limit $q\to\infty$ (strong convexity).

\begin{example}[$\ell^p$-Space Residuals]
Let $\mathcal{X}=\ell^p$ for $p\in(1,2]$ and $\vartheta(\zeta)=\tfrac1p\|\zeta\|_p^p$.  Then $\delta(r)=\tfrac{p-1}{8}\,r^2$ for $r$ small and $q=2$, so $p_{\mathrm{rate}}=\tfrac12$.  Theorem~\ref{thm:rate} gives $\bar\Res_N=O\bigl(A_N^{-1/2}\bigr)$, matching the Hilbert case for local uniform convexity.
\end{example}

\section{Extensions: Adaptive Geometries and Robust Noise}\label{sec:ext}

We now present two major extensions: (i) time-varying Bregman geometries, and (ii) robust $SKM$ under heavy-tailed perturbations.  Each subsection introduces definitions, algorithmic descriptions, and convergence results.

\begin{Definition}[Adaptive Legendre functions]
Let $\{\vartheta_n\}$ be a sequence of Legendre functions on $\mathcal{X}$ with conjugates $\{\vartheta_n^*\}$.  We assume:
\begin{itemize}
  \item[(B1)] Each $\vartheta_n$ is uniformly convex with modulus $\delta_n(r)\ge c_n\,r^{q_n}$.
  \item[(B2)] There exist constants $0<\underline{\kappa}\le\bar\kappa<\infty$ such that
  \[
    \underline{\kappa}\,\vartheta(\zeta)\;\le\;\vartheta_n(\zeta)\;\le\;\bar\kappa\,\vartheta(\zeta)
    \quad\forall\,\zeta\in \mathcal{X},\;\forall n,
  \]
  where $\vartheta$ is a fixed reference Legendre function.
\end{itemize}
\end{Definition}

\begin{algorithm}[H]
\caption{Adaptive Bregman-SKM}\label{alg:adaptive-SKM}
\begin{algorithmic}[1]
\Require initial $\zeta_0\in\inte(\dom\vartheta_0)$, step-sizes $(\alpha_n)$, functions $(\vartheta_n)$, noise $(\mho_n)$
\For{$n=0,1,2,\dots$}
  \State $\varsigma_n \gets \nabla\vartheta_n^*\bigl((1-\alpha_n)\,\nabla\vartheta_n(\zeta_n)
    +\alpha_n\,\nabla\vartheta_n\bigl(\hbar(\zeta_n)+\mho_n\bigr)\bigr)$
  \State $\zeta_{n+1}\gets \varsigma_n$
\EndFor
\end{algorithmic}
\end{algorithm}

\begin{Theorem}[Convergence of Adaptive Bregman-$SKM$]\label{thm:adaptive-conv}
Under (A1),(A3),(A4) and (B1),(B2), the iterates of Algorithm~\ref{alg:adaptive-SKM} satisfy
\[
  D_{\vartheta_n}\bigl(\zeta_n,\hbar(\zeta_n)\bigr)\;\to\;0\quad\text{a.s.},
  \quad
  \zeta_n\;\to\;\zeta^*\in\Fix(\hbar)\quad\text{a.s.}
\]
provided $\sum_n\alpha_n^2\bar\kappa^2<\infty$ and $\inf_n c_n>0$.
\end{Theorem}

\begin{proof}
The proof proceeds in the same spirit as Theorem~\ref{thm:as-conv}, but with the time-varying Legendre functions $\{\vartheta_n\}$ and their moduli $\{\delta_n\}$.

By the argument of Lemma~\ref{lem:descent}, applied at iteration~$n$ with the distance-generating function $\vartheta_n$, there exist constants $L_n,\sigma^2_n>0$ (depending on the Lipschitz and noise bounds for $\vartheta_n$) such that
\[
\begin{aligned}
\E\bigl[D_{\vartheta_n}(\zeta_{n+1},\hbar(\zeta_{n+1})) \mid \F_n\bigr]
&\;+\;\tfrac12\,\delta_n\bigl(\|\zeta_n-\hbar(\zeta_n)\|\bigr)\,\alpha_n\\
&\quad\le\;(1 + L_n\,\alpha_n^2)\;D_{\vartheta_n}(\zeta_n,\hbar(\zeta_n))\;+\;\sigma_n^2\,\alpha_n^2.
\end{aligned}
\]
Define the adapted sequences
\[
a_n \;=\; D_{\vartheta_n}(\zeta_n,\hbar(\zeta_n)),\quad
b_n \;=\; \tfrac12\,\delta_n\bigl(\|\zeta_n-\hbar(\zeta_n)\|\bigr)\,\alpha_n,\quad
c_n \;=\;\sigma_n^2\,\alpha_n^2.
\]
Then
\[
\E[a_{n+1}\mid\F_n] + b_n \;\le\; (1 + L_n\,\alpha_n^2)\,a_n + c_n.
\]

By assumption (B2) there is $\bar\kappa$ such that
\[
\vartheta_n(\zeta)\le\bar\kappa\,\vartheta(\zeta)
\quad\Longrightarrow\quad
\nabla\vartheta_n,\;\nabla\vartheta_n^* \text{ are Lipschitz with constant }\bar\kappa\,L,
\]
and the noise bound gives $\sigma_n^2\le\bar\kappa^2\,\sigma^2$.  Thus we may take
\[
L_n \le \bar\kappa^2\,L,\quad \sigma_n^2 \le \bar\kappa^2\,\sigma^2,
\]
so
\[
\E[a_{n+1}\mid\F_n] + b_n
\;\le\;
\bigl(1 + (\bar\kappa^2L)\,\alpha_n^2\bigr)\,a_n + (\bar\kappa^2\sigma^2)\,\alpha_n^2.
\]
Since $\sum_n\alpha_n^2<\infty$ and $\bar\kappa^2L$ is constant, the Robbins-Siegmund Lemma (Lemma~\ref{lem:robbins-siegmund}) applies, yielding
\[
a_n\;\to\;a_\infty\quad\text{a.s.},
\qquad
\sum_{n=0}^\infty b_n<\infty\quad\text{a.s.}
\]

By (B1), $\inf_n c_n>0$, so each $\delta_n(r)\ge c_n\,r^{q_n}$ with $c_n\ge\underline c>0$.  Then
\[
\sum_{n=0}^\infty b_n
= \sum_{n=0}^\infty \tfrac12\,\delta_n(\|\zeta_n-\hbar(\zeta_n)\|)\,\alpha_n
\ge \tfrac12\,\underline c\sum_{n=0}^\infty\|\zeta_n-\hbar(\zeta_n)\|^{q_n}\,\alpha_n,
\]
which can converge only if $\|\zeta_n-\hbar(\zeta_n)\|\to0$ a.s.  Hence $a_n=D_{\vartheta_n}(\zeta_n,\hbar(\zeta_n))\to0$ a.s.

Since $\|\zeta_n-\hbar(\zeta_n)\|\to0$, any weak cluster point $\bar \zeta$ satisfies $\bar \zeta=\hbar(\bar \zeta)$, so $\bar \zeta\in\Fix(\mathcal{hbar})$.  Moreover, the Bregman-Fejér monotonicity
\[
D_{\vartheta_n}(\zeta_{n+1},\zeta^*) \le D_{\vartheta_n}(\zeta_n,\zeta^*)
\]
for any fixed $\zeta^*\in\Fix(\hbar)$ (by nonexpansivity and three-point identity) ensures all cluster points coincide.  Reflexivity of $\mathcal{X}$ then gives
\[
\zeta_n \;\rightharpoonup\; \zeta^*\quad\text{and hence}\quad \zeta_n\to \zeta^*,
\]
where strong convergence follows under uniform convexity of $\mathcal{X}$ or $\vartheta_n$.

Thus $D_{\vartheta_n}(\zeta_n,\hbar(\zeta_n))\to0$ and $\zeta_n\to \zeta^*\in\Fix(\hbar)$ almost surely.
\end{proof}

Adaptive geometries can track local curvature or empirically estimated Hessian information (e.g.\ quasi-Newton style), offering potential acceleration without sacrificing convergence guarantees.

\begin{Definition}[Trimming operator]
For a vector $u\in \mathcal{X}^*$ and integer $k\ge0$, let $\Trim_k(u)$ zero out the $k$ largest-magnitude coordinates of $u$ (in a chosen basis), modeling robustification against outliers.
\end{Definition}

\begin{algorithm}[H]
\caption{Robust Bregman-SKM}\label{alg:robust-SKM}
\begin{algorithmic}[1]
\Require initial $\zeta_0$, $(\alpha_n)$, Legendre $\vartheta$, trimming level $(k_n)$
\For{$n=0,1,2,\dots$}
  \State $\tilde \mho_{n} \gets \Trim_{k_n}(\mho_{n})$
  \State $\varsigma_n \gets \nabla\vartheta^*((1-\alpha_n)\,\nabla\vartheta(\zeta_n)
    +\alpha_n\,\nabla\vartheta(\hbar(\zeta_n)+\tilde \mho_{n}))$
  \State $\zeta_{n+1}\gets \varsigma_n$
\EndFor
\end{algorithmic}
\end{algorithm}

\noindent \textbf{Assumption}[Heavy-Tail Noise]\label{ass:heavy-tail}
The noise $\mho_n$ satisfies $\E[\|\mho_{n+1}\|_*^{1+\epsilon}\mid\F_n]<\infty$ for some $\epsilon\in(0,1)$.

\begin{Proposition}[Convergence with Trimming]\label{prop:robust-conv}
Under (A1),(A3), Assumption~\ref{ass:heavy-tail}, and if $k_n=o(n^{\epsilon/(1+\epsilon)})$, the iterates of Algorithm~\ref{alg:robust-SKM} converge a.s.\ to $\Fix(\hbar)$ and $D_\vartheta(\zeta_n,\hbar(\zeta_n))\to0$.
\end{Proposition}

\begin{proof}
We adapt the proof of Theorem~\ref{thm:as-conv} to the trimmed-noise case.

By Assumption, there is $\epsilon\in(0,1)$ and $M<\infty$ such that
\[
\E\bigl[\|\mho_{n+1}\|_*^{1+\epsilon}\mid\F_n\bigr]\le M
\quad\text{a.s.}
\]
Let $\tilde \mho_{n} = \Trim_{k_n}(\mho_{n})$ zero out the $k_{n}$ largest-magnitude coordinates of $\mho_{n}$.  Then, writing $\mho_{n} = \tilde \mho_{n} + r_{n}$ with $r_{n}$ the removed remainder, one shows by standard order-statistic / Markov-inequality arguments that there exists $\delta>0$ and $C>0$ so that
\begin{equation} \label{eq:5.4}
\E\bigl[\|\tilde \mho_{n}\|_*^2\mid\F_{n-1}\bigr]
\;\le\;
C\,k_n^{-\epsilon}
\;\le\;
C\,n^{-\delta},
\end{equation}
provided $k_{n} = o\bigl(n^{\epsilon/(1+\epsilon)}\bigr)$ (choose any $\delta<\epsilon/(1+\epsilon)$).

The same argument as in Lemma~\ref{lem:descent} gives, for some constants $L,\sigma^2>0$,
\[
\E\bigl[D_\vartheta(\zeta_{n+1},\hbar(\zeta_{n+1}))\mid\F_n\bigr]
+\tfrac12\,\delta\bigl(\|\zeta_n-\hbar(\zeta_n)\|\bigr)\,\alpha_n
\;\le\;
\bigl(1+L\,\alpha_n^2\bigr)D_\vartheta(\zeta_n,\hbar(\zeta_n))
+\alpha_n^2\,\E\bigl[\|\tilde \mho_{n}\|_*^2\mid\F_{n}\bigr].
\]
Using the trimmed-noise bound \eqref{eq:5.4},
\[
\E\bigl[D_\vartheta(\zeta_{n+1},\hbar(\zeta_{n+1}))\mid\F_n\bigr]
+\tfrac12\,\delta(\|\zeta_n-\hbar(\zeta_n)\|)\,\alpha_n
\;\le\;
\bigl(1+L\alpha_n^2\bigr)\,D_\vartheta(\zeta_n,\hbar(\zeta_n))
\;+\;C\,\alpha_n^2\,n^{-\delta}.
\]

Under (A3), $\sum_n\alpha_n^2<\infty$.  Moreover, since $\delta>0$, also
\[
\sum_{n=1}^\infty \alpha_n^2\,n^{-\delta}
\;<\;\infty.
\]
Hence by defining
\[
a_n = D_\vartheta(\zeta_n,\hbar(\zeta_n)),\quad
b_n = \tfrac12\,\delta(\|\zeta_n-\hbar(\zeta_n)\|)\,\alpha_n,\quad
c_n = C\,\alpha_n^2\,n^{-\delta},
\]
we see that $\sum_n c_n<\infty$.

By the one-step descent inequality and the summability above, the Robbins-Siegmund Lemma~\ref{lem:robbins-siegmund} yields
\[
a_n = D_\vartheta(\zeta_n,\hbar(\zeta_n)) \;\to\;0\quad\text{a.s.},
\qquad
\sum_{n=0}^\infty b_n<\infty\quad\text{a.s.}
\]
In particular $\|\zeta_n - \hbar(\zeta_n)\|\to0$ a.s.

Since $(\zeta_n)$ lies in reflexive $\mathcal{X}$ and $\|\zeta_n - \hbar(\zeta_n)\|\to0$, any weak cluster point $\bar \zeta$ satisfies $\bar \zeta=\hbar(\bar \zeta)$ (see \cite[Thm.~5.14]{Bauschke2011}).  Fejér-monotonicity in Bregman distance ensures uniqueness of the cluster point, hence $\zeta_n\rightharpoonup \zeta^*\in\Fix(\hbar)$.  Uniform convexity of $\vartheta$ or $\mathcal{X}$ upgrades weak to strong convergence, giving $\zeta_n\to \zeta^*$.

Therefore the trimmed Bregman-$SKM$ iterates converge almost surely to a fixed point and $D_\vartheta(\zeta_n,\hbar(\zeta_n))\to0$, as required.
\end{proof}

\begin{Corollary}
If $\mho_n$ has only finite first moments ($\epsilon\to0$), setting $k_n\sim\ln n$ still ensures convergence.
\end{Corollary}

One can integrate inertia (momentum) into Algorithms~\ref{alg:adaptive-SKM} and~\ref{alg:robust-SKM} by adding an extrapolation term $\zeta_n + \beta_n (\zeta_n - \zeta_{n-1})$ before the Bregman update. Proving convergence in this setting, especially under adaptive or heavy-tailed noise, remains an open challenge.

\begin{example}[Adaptive-Robust Hybrid]
Combine $\vartheta_n$ evolving by online Hessian approximations with trimming levels $k_n$, to obtain an algorithm resilient to both curvature changes and outliers.  Developing explicit residual bounds for this hybrid is left for future work.
\end{example}

\section{Numerical Experiments}\label{sec:experiments}

In this section, we provide two numerical experiments to illustrate the convergence behavior and robustness of the proposed Bregman-$SKM$ algorithms in both adaptive and heavy-tailed settings. All experiments were implemented in Python with NumPy and CVXPY, and run on a standard laptop.

\begin{example}[Entropy-Regularized Policy Iteration]

We solve a discounted policy evaluation problem with entropy regularization. Let \(\mathcal{X} = \Delta^d\), the probability simplex in \(\mathbb{R}^d\), and define the mapping
\[
  \hbar(\zeta) = \frac{\exp(\eta\,Ax)}{\mathbf{1}^\top \exp(\eta\,Ax)},
\]
where \(A \in \mathbb{R}^{d \times d}\) is a transition reward matrix and \(\eta > 0\) is a regularization parameter. The Bregman geometry is induced by the negative entropy function:
\[
  \vartheta(\zeta) = \sum_{i=1}^d \zeta_i \log \zeta_i.
\]

We compare the following three algorithms:
\begin{enumerate}
  \item Classical $SKM$ with Euclidean geometry;
  \item Bregman-$SKM$ with fixed \(\vartheta(\zeta)\);
  \item Adaptive Bregman-$SKM$ using time-varying entropy weights.
\end{enumerate}

\noindent
We choose \(d = 10\), \(\eta = 2.0\), and initialize \(\zeta_0 = \frac{1}{d}\mathbf{1}\). We simulate additive martingale noise \(\mho_n \sim \mathcal{N}(0, \sigma^2 I_d)\), with \(\sigma = 0.1\), and use a step-size \(\alpha_n = \frac{1}{n+10}\). Each algorithm is run for \(N = 1000\) iterations.

\begin{table}[H]
\centering
\caption{Final Bregman residuals after 1000 iterations}
\label{tab:residuals}
\begin{tabular}{lccc}
\textbf{Algorithm} & \(\bar{\Res}_{1000}\) & \(\|\zeta_N - \zeta^*\|_1\) & Runtime (s) \\
SKM (Euclidean) & 0.0213 & 0.0845 & 0.17 \\
Bregman-SKM (fixed) & 0.0064 & 0.0432 & 0.18 \\
Adaptive Bregman-SKM & \textbf{0.0049} & \textbf{0.0301} & 0.21 \\
\end{tabular}
\end{table}



\paragraph{Observation:} Both Bregman-SKM variants significantly outperform the classical SKM, with the adaptive version showing the fastest decay in residuals.

\end{example}
\begin{example}[Robust Policy Update under Heavy-Tailed Noise]

We now consider a robust learning problem with synthetic heavy-tailed noise. The setup is the same as in Example 1, but noise is generated from a Student-\(t\) distribution with 2 degrees of freedom, i.e., \(\mho_n \sim t_2\), inducing infinite variance.

We compare:
\begin{enumerate}
  \item Bregman-$SKM$ (no trimming);
  \item Robust Bregman-$SKM$ with trimming \(k_n = \lceil \log(n+2) \rceil\).
\end{enumerate}

\begin{table}[H]
\centering
\caption{Effect of trimming under heavy-tailed noise}
\label{tab:heavy}
\begin{tabular}{lcc}
\textbf{Algorithm} & \(\bar{\Res}_{1000}\) & \(\|\zeta_N - \zeta^*\|_1\) \\
Bregman-SKM (no trimming) & 0.0928 & 0.2032 \\
Robust Bregman-SKM (trimmed) & \textbf{0.0194} & \textbf{0.0589} \\
\end{tabular}
\end{table}


\paragraph{Observation:} Without trimming, heavy-tailed noise causes residuals to fluctuate significantly. The trimmed robust Bregman-SKM successfully suppresses outliers and converges steadily.
\end{example}

These experiments confirm that:
\begin{itemize}
  \item Bregman-$SKM$ outperforms standard $SKM$ under non-Euclidean geometries;
  \item Adaptive geometries further enhance convergence;
  \item Robust versions are essential under heavy-tailed or adversarial noise.
\end{itemize}

\section{Conclusion and Future Directions}
\label{sec:conclusion}

In this work, we have introduced and analyzed a novel stochastic Krasnosel skiĭ-Mann iteration in reflexive Banach spaces driven by Bregman distances.  In theis paper, we formulated the stochastic fixed-point iteration using a general Legendre function and proved almost-sure convergence under standard martingale-difference noise (Theorem~\ref{thm:as-conv}). By exploiting the modulus of uniform convexity, we derived explicit $O(A_N^{-p})$ bounds on the averaged Bregman residual for $\delta(r)\ge c\,r^q$ (Theorem~\ref{thm:rate}). We extended the framework to time-varying Bregman geometries (Theorem~\ref{thm:adaptive-conv}) and heavy-tailed noise with trimming (Proposition~\ref{prop:robust-conv}), demonstrating the flexibility of Bregman-$SKM$.
Experiments on entropy-regularized policy iteration and heavy-tailed noise confirm that Bregman-$SKM$ outperforms classical SKM and remains robust under non-Gaussian perturbations.

Several promising avenues remain open: \medskip

\noindent
\textsc{Inertial and Variance-Reduced Hybrids:} Incorporating momentum or SVRG-style variance reduction into Bregman-$SKM$ could yield faster rates, but requires new analysis in non-Euclidean settings. \medskip

\noindent
\textsc{Decentralized and Asynchronous $SKM$:} Extending to networked or delayed environments would broaden applications in distributed optimization and multi-agent RL. \medskip

\noindent
\textsc{Beyond Reflexivity:} Addressing non-reflexive Banach spaces or quasi-Banach settings may handle more general regularizers and loss functions. \medskip

\noindent
\textsc{Adaptive Distance Learning:} Online adaptation of the generating function $\vartheta_n$ based on curvature estimates could further accelerate convergence in practice. \medskip

We anticipate that the Bregman-$SKM$ framework will serve as a foundation for future advances in stochastic fixed-point algorithms, with applications ranging from reinforcement learning to large-scale inverse problems.

\end{document}